[1994/06/01]

\documentclass[12pt,reqno,intlimits,a4paper,english]{amsart}

\usepackage{pstricks,graphicx}
\usepackage{calc}
\usepackage{babel}
\usepackage{amssymb}

\numberwithin{equation}{section}
\newcounter{hours}\newcounter{minutes}

\newlength{\Indent}
\newlength{\Parskip}
\theoremstyle{plain}
\newtheorem{thm}{Theorem}[section]     
\newtheorem{prop}[thm]{Proposition}

\theoremstyle{remark}

\newtheorem{Remark}[thm]{Remark}

\theoremstyle{definition}
\newtheorem{Defin}[thm]{Definition}
\newenvironment{defin}{\begin{Defin}}{\qed\end{Defin}}

\DeclareMathOperator{\Div}{div}
\DeclareMathOperator{\Lapl}{\Delta}
\DeclareMathOperator{\dist}{dist}

\newcommand{\RN}{\Bbb{R}^{N}}
\newcommand{\R}{\Bbb{R}}
\newcommand{\abs}[1]{\lvert#1\rvert}

\newcommand{\di}{\,\text{\rmfamily\upshape d}}

\newcommand{\pder}[2]{\frac{\partial #1} {\partial #2}}

\newcommand{\Om}{\varOmega}
\newcommand{\eps}{\varepsilon}

\newcommand{\qefed}{=:}

\newcommand{\CC}{\mathcal{C}}
\newcommand{\ZZ}{\Bbb{Z}}

\def\X{{\mathcal X}}
\newcommand{\dfint}{\mu_{\textup{int}}}
\newcommand{\dfout}{\mu_{\textup{out}}}
\newcommand{\dfboth}{\mu}
\newcommand{\dfbothe}{\mu^\eps}
\newcommand{\gammafinale}{\widetilde\dfboth}
\newcommand{\lfint}{\lambda_{\textup{int}}}
\newcommand{\lfout}{\lambda_{\textup{out}}}
\newcommand{\lfboth}{\lambda}
\newcommand{\lfbothe}{\lambda^\eps}
\newcommand{\lfav}{\lfboth_{0}}

\newcommand{\eh}{\boldsymbol{e}_{h}}
\newcommand{\ej}{\boldsymbol{e}_{j}}

\newcommand{\reps}{r_{\eps}}
\newcommand{\rteps}{\tilde r_{\eps}}

\newcommand{\Per}{E}
\newcommand{\Omint}{\Om_{\textup{int}}^{\eps}}
\newcommand{\Omout}{\Om_{\textup{out}}^{\eps}}
\newcommand{\Memb}{\varGamma^{\eps}}

\newcommand{\Perint}{\Per_{\textup{int}}}
\newcommand{\Perout}{\Per_{\textup{out}}}
\newcommand{\Permemb}{\varGamma}

\newcommand{\intotom}{\int_{0}^{t}\!\!\int_{\Om}}

\newcommand{\intotme}{\int_{0}^{t}\!\!\int_{\Memb}}
\newcommand{\oldfeps}{E_{\eps}}
\newcommand{\XX}{\X^\eps}
\newcommand{\beltrami}{\Delta^{\!\!B}}
\newcommand{\beltramigrad}{\nabla^{B}}
\newcommand{\beltramidiv}{\Div^{B}}

\newcommand{\cellproblembasename}{\boldsymbol{\mathcal{P}}}
\newcommand{\dfmat}{A^{hom}}
\newcommand{\cpo}{\cellproblembasename_{0}}
\newcommand{\cpi}{\cellproblembasename_{1}}
\newcommand{\cpii}{\cellproblembasename_{2}}
\newcommand{\spaziosoleps}{L^2\big(0,T;\XX_0(\Om)\big)}

\newcommand{\const}{\gamma}

\makeindex
\begin{document}

\title
{Error estimate for a homogenization problem involving the Laplace-Beltrami operator}
\author{M. Amar$^*$ -- R. Gianni$^\ddag$\\
\hfill \\
$^*$Dipartimento di Scienze di Base e Applicate per l'Ingegneria\\
Sapienza - Universit\`a di Roma\\
Via A. Scarpa 16, 00161 Roma, Italy\\ \\
$^\ddag$Dipartimento di Matematica ed Informatica\\
   Universit\`{a} di Firenze\\
   Via Santa Marta 3, 50139 Firenze, ITALY
}

\begin{abstract}
In this paper we prove an error estimate for a model of heat conduction in composite materials having
a microscopic structure arranged in a periodic array and thermally active membranes separating
the heat conductive phases.
\medskip

  \textsc{Keywords:} Homogenization, Asymptotic expansion, Laplace-Beltrami operator, Heat conduction.

  \textsc{AMS-MSC:} 35B27, 35Q79
  \bigskip

  \textbf{Acknowledgments}: We would like to thank R. Lipton and P. Bisegna for some helpful discussions. The first author is member of the \emph{Gruppo Nazionale per l'Analisi Matematica, la Probabilit\`{a} e le loro Applicazioni} (GNAMPA) of the \emph{Istituto Nazionale di Alta Matematica} (INdAM).

\end{abstract}

\maketitle

\ifx\Versione\UnDeFiNeD\else
\begin{center}
\Versione
\end{center}
\fi


\section{Introduction}\label{s:introduction}
Heat and electrical conduction in composite materials has been widely investigated
in the last years in the context of homogenization theory (see among others, e.g.
\cite{Amar:Andreucci:Bellaveglia:2015A,Amar:Andreucci:Bellaveglia:2015B,
Amar:Andreucci:Bisegna:Gianni:2003c,Amar:Andreucci:Bisegna:Gianni:2004a,
Amar:Andreucci:Bisegna:Gianni:2006a,Amar:Andreucci:Bisegna:Gianni:2010,Auriault:Ene:1994,
Timofte:2016,Donato:Monsurro:2004,Hummel:2000,
Jose:2009,Timofte:2013}).
In this paper we will focus on the study of models of heat conduction in composite materials used for
encapsulation of electronic devices. This topic is attracting increasing interest among researchers,
both from the point of view of applications and also in a more mathematical setting. In our
previous paper \cite{Amar:Gianni:2016A} (to which we refer for a more detailed physical description
of the problem) a composite medium was taken into account, which was made of a hosting material
with inclusions separated from their surroundings by a thermally active membrane.

Such a situation is consistent with many physical applications in which a material must be modified in a way such that its thermal conductivity is enhanced while preserving other material properties e.g. ductility. This is, as stated above, the case of polymer encapsulation of electronic devices as well as, just to make an example, engine coolants. Specifically, in the first case, ductility of the material is required to fill the voids and the interstices among the electrical components by applying a moderate pressure. Polymers and rubbers have this property but they do not display a satisfactory heat dissipation which, on the other hand, can be attained by adding highly conductive nanoparticles. In some situations, these nanoparticles are enclosed in a membrane separating them from the surrounding medium. It is therefore only natural to investigate the influence of these membranes on the overall conductivity of the composite medium under different assumptions on the thermal behaviour of these interfaces. The case of perfect or imperfect thermal contact, though interesting from the point of view of applications, is mathematically well known, for this reason we focused on the case in which the membrane is thermally active, e.g. a tangential heat diffusion takes place. In \cite{Amar:Gianni:2016A} a macroscopic model was deduced, via the unfolding homogenization technique, assuming
the periodicity of the microscopic structure, whose characteristic length is described
by a small parameter $\eps$. We make use of a sensible mathematical description of the behavior
of the interfaces which are modeled by means of the Laplace-Beltrami operator
(see, e.g. \cite{Allaire:Damlamian:Hornung:1995,Andreucci:Bisegna:Dibenedetto:2003}).

In this paper we complete the research started in \cite{Amar:Gianni:2016A} providing an ``error estimate"
which enables us to evaluate
the rate of convergence, with respect to $\eps\to0$, of the solution $u_{\eps}$ of the microscopic
(physical) problem to the solution $u_{0}$ of the macroscopic one. More precisely, we prove
\begin{equation*}
\begin{aligned}
&  \Vert u_{\eps}-(u_{0}+\eps u_1)\Vert_{L^2(0,T;H^1(\Om))} &\le \const \sqrt\eps\,,
\\
&  \Vert u_{\eps}-u_{0}\Vert_{L^2(\Om_T)} &\le \const \sqrt\eps\,,
 \end{aligned}
\end{equation*}
for a proper constant $\const>0$ independent of $\eps$, where $u_1$ is the so called first corrector and it is
defined in \eqref{eq:u0factor}.

To obtain this estimate we follow the classical approach given by the asymptotic
expansions due to Bensoussan-Lions-Papanicolaou \cite{Bensoussan:Lions:Papanicolaou:1978}
which, under extra-regularity assumptions, gives an $H^1$-estimate for this error.
The knowledge of the rate of convergence is a crucial tool for numerical applications.
Moreover, we prove the symmetry and the strict positivity of the matrix describing the diffusivity
of the macroscopic (homogenized) material. This last result is crucial to guarantee the
well-posedness of the parabolic limit equation.

Though the results proved in this paper are along the same lines of other ones obtained in the framework of
the homogenization theory, nevertheless they are of some mathematical interest due to the presence
of the Laplace-Beltrami operator, which makes the computations a bit tricky.
\medskip

The paper is organized as follows. In Section \ref{s:threeD_problem} we recall the definition and
some properties of the tangential operators (gradient, divergence, Laplace-Beltrami operator),
we state our geometrical setting and present our model.
In Section \ref{s:homog}, after having proved some energy inequalities, we follow the formal
approach by Bensoussan-Lions-Papanicolau in order to introduce the cell functions and to guess the limit equation,
proving the ellipticity of its principal part (see Theorem \ref{l:l1}).
Finally, in Section \ref{s:error} taking advantage of the asymptotic expansions obtained in Section \ref{s:homog}, we
provide the error estimate (see Theorem \ref{t:t1}).

\section{Preliminaries}\label{s:threeD_problem}

\subsection{Tangential derivatives}
\label{s:LB}
Let $\phi$ be a ${\mathcal C}^2$-function,  $\mathbf{\Phi}$ be a
${\mathcal C}^2$-vector function and $S$ a smooth surface
with normal unit vector $n$.
We recall that the tangential gradient of $\phi$ is given by
\begin{equation}\label{eq:a5bis}
\beltramigrad\phi=\nabla\phi-(n\cdot\nabla\phi)n
\end{equation}
and the tangential divergence of $\mathbf{\Phi}$ is given by
\begin{multline}\label{eq:a3bis}
\beltramidiv\mathbf{\Phi}
=\Div\mathbf{\Phi}- (n\cdot\nabla\mathbf{\Phi}_i)n_i-(\Div n)(n\cdot\mathbf{\Phi})
\\
=\beltramidiv \left(\mathbf{\Phi}- (n\cdot\mathbf{\Phi})n\right)
=\Div \left(\mathbf{\Phi}- (n\cdot\mathbf{\Phi})n\right)\,,
\end{multline}
where, taking into account the smoothness of $S$,
the normal vector $n$ can be naturally defined in a small neighborhood of $S$ as
$\frac{\nabla d}{|\nabla d|}$, where $d$ is the signed distance from $S$.
Moreover, we define the Laplace-Beltrami operator as
\begin{equation}
\label{eq:beltrami1}
\beltrami\phi =\Div^B(\beltramigrad\phi)\,,
\end{equation}
so that, by \eqref{eq:a5bis} and \eqref{eq:a3bis}, we get that the Laplace-Beltrami operator can be written as
\begin{multline}
\label{eq:beltrami}
\beltrami\phi =
\Delta\phi-n^t\nabla^2\phi n - (n\cdot\nabla\phi)\Div n\\
= (\delta_{ij}-n_in_j)\partial^2_{ij}\phi -  n_j\partial_j\phi \partial_in_i
= (Id-n\otimes n)_{ij}\partial^2_{ij}\phi-(n\cdot\nabla\phi)\Div n\,,
\end{multline}
where $\nabla^2\phi$ stands for the Hessian matrix of $\phi$.
Finally, we recall that on a regular surface $S$ with no boundary (i.e. when $\partial S=\emptyset$)
we have
\begin{equation}\label{eq:a66}
\int_S \beltramidiv \mathbf{\Phi}\di\sigma =0 \,.
\end{equation}

\subsection{Geometrical setting}\label{ss:geometric}
The typical periodic geometrical setting is displayed in Figure~\ref{fig:omega}.
Here we give, for the sake of clarity, its detailed formal definition.


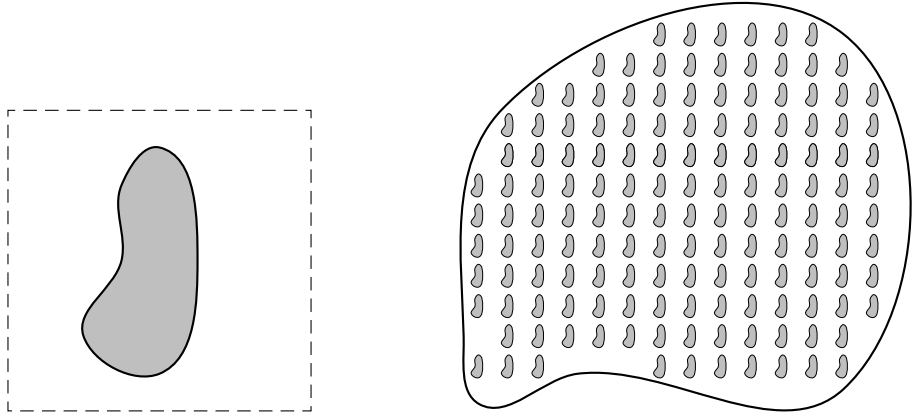
\begin{figure}[htbp]%
\begin{center}%
\begin{pspicture}(12,6)
\rput(0,1){
\psframe[linewidth=0.4pt,linestyle=dashed](0,0)(4,4)
\psccurve[fillstyle=solid,fillcolor=lightgray](2,.5)(2.5,2)(2,3.5)(1.5,3)(1.5,2)(1,1)
}
\newcommand{\minicell}{\scalebox{0.1}{
\psccurve[fillstyle=solid,fillcolor=lightgray](2,.5)(2.5,2)(2,3.5)(1.5,3)(1.5,2)(1,1)
}}
\rput(6,1){
\multirput(0,0.4)(0.4,0){3}{\minicell}%
\multirput(2.4,0.4)(0.4,0){7}{\minicell}%
\multirput(0.4,0.8)(0.4,0){12}{\minicell}%
\multirput(0,1.2)(0,0.4){5}{\multirput(0,0)(0.4,0){14}{\minicell}}%
\multirput(0.4,3.2)(0.4,0){13}{\minicell}%
\multirput(0.4,3.2)(0.4,0){13}{\minicell}%
\multirput(0.4,3.6)(0.4,0){13}{\minicell}%
\multirput(0.8,4)(0.4,0){12}{\minicell}%
\multirput(1.6,4.4)(0.4,0){9}{\minicell}%
\multirput(2.4,4.8)(0.4,0){6}{\minicell}%
\psccurve(0.2,.1)(1.5,.5)(5,.3)(5,5)(.5,4)(0,1)
}

\end{pspicture}%
    \caption{\textsl{Left}: the periodic cell $Y$. $\Perint$ is the shaded
    region and $\Perout$ is the white region.
    \textsl{Right}: the region $\Om$.}
    \label{fig:omega}
  \end{center}
\end{figure}

Let us introduce a periodic open subset $\Per$
of $\RN$, so that $\Per+z=\Per$ for all $z\in\ZZ^{N}$.
We employ the notation $Y=(0,1)^{N}$, and
$\Perint=\Per\cap Y$, $\Perout=Y\setminus\overline{\Per}$,
$\Permemb=\partial\Per\cap \overline Y$. As a simplifying assumption,
we stipulate that $|\Permemb\cap\partial Y|_{N-1}=0$.

Let $\Om$ be an open connected bounded subset of $\RN$; for all
$\eps>0$ define $\Omint=\Om\cap\eps \Per$,
$\Omout=\Om\setminus\overline{\eps \Per}$, so that
$\Om=\Omint\cup\Omout\cup\Memb$, where $\Omint$ and $\Omout$ are
two disjoint open subsets of $\Om$, and
$\Memb=\partial\Omint\cap\Om=\partial\Omout\cap\Om$.
The region $\Omout$ [respectively, $\Omint$] corresponds to the outer
phase [respectively, the inclusions], while
$\Memb$ is the interface.
We assume also that $\Om$ and $\Per$ have regular boundary and we
stipulate that $\dist(\Memb,\partial\Om)\geq \gamma_0\eps$,
for a suitable $\gamma_0>0$.
To this purpose, for each $\eps$,
we are ready to remove the inclusions in
all the cells which are not
completely contained in $\Om$ (see Figure \ref{fig:omega}).
This assumption is in accordance with our previous papers (see \cite{Amar:Andreucci:Bisegna:Gianni:2004a, Amar:Andreucci:Bisegna:Gianni:2006a, Amar:Andreucci:Bisegna:Gianni:2009, Amar:Andreucci:Bisegna:Gianni:2009a, Amar:Andreucci:Bisegna:Gianni:2010, Amar:Andreucci:Bisegna:Gianni:2013}) and maybe it can be dropped as in \cite{Allaire:Murat:1993, Cioranescu:Damlamian:Donato:Griso:Zaki:2012};
nevertheless we will not pursue this line of investigation in this paper.

Moreover, let $\nu$ denote the normal unit vector to $\Permemb$ pointing
into $\Perout$, extended by periodicity to the whole of $\R^N$, so that $\nu_\eps(x)=\nu(x/\eps)$
denotes the normal unit vector to $\Memb$ pointing into $\Omout$.

Finally, given $T > 0$, we denote by $\Om_T = \Om \times (0, T)$. More in general, for any spatial domain $G$, we denote by $G_T = G\times(0, T)$.
\medskip

\subsection{Position of the problem}\label{ss:position}
Let $\dfbothe, \lfbothe:\Om\to\R$ be defined as
\begin{equation*}
\begin{aligned}
&  \lfbothe=\lfint\,\quad \text{in $\Omint$,}\qquad \lfbothe=\lfout\,\quad
  \text{in $\Omout$;}
 \\
 & \dfbothe=\dfint\,\quad \text{in $\Omint$,}\qquad \dfbothe=\dfout\,\quad
  \text{in $\Omout$.}
\end{aligned}
\end{equation*}
For every $\eps>0$, we consider the problem
for $u_{\eps}(x,t)$ given by
\begin{alignat}2
  \label{eq:PDEin}
  \dfbothe \pder{u_\eps}{t}-\Div(\lfbothe \nabla u_{\eps})&=0\,,&\qquad &\text{in $\Om_T$;}
  \\
  \label{eq:FluxCont}
  [u_{\eps}] &=0 \,,&\qquad &\text{on
  $\Memb_T$;}
\end{alignat}
 \begin{alignat}2
  \label{eq:Circuit}
  \eps{\alpha} \pder{u_{\eps}}{t}-\eps\beta\beltrami u_\eps
  &=[\lfbothe \nabla u_\eps  \cdot \nu_\eps]\,,&\qquad
  &\text{on $\Memb_T$;}
   \\
 \label{eq:BoundData}
  u_{\eps}(x,t)&=0\,,&\qquad&\text{on $\partial\Om\times(0,T)$;}
  \\
  \label{eq:InitData}
  u_{\eps}(x,0)&=\overline u_0(x)\,,&\qquad&\text{in $\Om$,}
\end{alignat}
where we denote
\begin{equation}
  \label{eq:jump}
  [u_{\eps}] = u_{\eps}^{\text{out}}  -  u_{\eps}^{\text{int}} \,,
\end{equation}
and the same notation is employed also for other quantities.
We assume that all the constants $\dfint, \dfout, \lfint, \lfout, \alpha, \beta$, involved in equations
\eqref{eq:PDEin} and \eqref{eq:Circuit} are strictly positive.

Since problem \eqref{eq:PDEin}--\eqref{eq:InitData} is not standard, in order to define a proper notion of weak
solution, we will need to introduce some suitable function spaces.
To this purpose and
for later use, we will denote by $H^1_B(\Memb)$ the space of Lebesgue measurable functions
$u:\Memb\to\R$ such that $u\in L^2(\Memb)$, $\beltramigrad u \in L^2(\Memb)$.
Let us also set
\begin{equation}
 \label{eq:space2}
 \XX_0(\Om) :=H^1_0(\Om)\cap H^1_B(\Memb)\,.
\end{equation}


\begin{defin}\label{d:weak_sol}
We say that $u_\eps\in \spaziosoleps$ is a weak solution of problem \eqref{eq:PDEin}--\eqref{eq:InitData} if
\begin{multline}\label{eq:weak_sol}
-\int_{0}^{T}\!\!\int_{\Om} \dfbothe u_{\eps}\pder{\phi}{\tau}\di x\di \tau+
  \int_{0}^{T}\!\!\int_{\Om} \lfbothe \nabla u_{\eps}\cdot\nabla\phi \di x\di \tau
  -{\eps\alpha}\int_0^T\!\!\int_{\Memb} u_{\eps }\pder{\phi}{\tau}\di\sigma\di\tau
  \\
  +{\eps\beta}\int_0^T\!\!\int_{\Memb} \beltramigrad u_{\eps}\cdot \beltramigrad \phi\di\sigma\di\tau
  =\int_{\Om} \dfbothe \overline u_0\phi(x,0)\di x
  +{\eps\alpha}\int_{\Memb} \overline u_0\phi(x,0)\di\sigma \,,
\end{multline}
for every test function $\phi\in \CC^\infty(\Om_T)$ such that $\phi$ has compact support in $\Om$ for every
$t\in(0,T)$ and $\phi(\cdot,T)=0$ in $\Om$.
\end{defin}

If $u_\eps$ is smooth, by \eqref{eq:beltrami} it follows that equation \eqref{eq:Circuit} can be written in the form
\begin{equation}
\label{eq:Circuitnew}
\eps{\alpha} \pder{u_{\eps}}{t}-\eps\beta\left( \Delta u_\eps- \nu_\eps^t\nabla^2 u_\eps\nu_\eps -
(\nu_\eps\cdot\nabla u_\eps)\Div \nu_\eps\right)
 =[\lfboth \nabla u_\eps  \cdot \nu_\eps]\,,\qquad  \text{on $\Memb$,}
\end{equation}
where, as in \eqref{eq:beltrami}, $\nabla^2u_\eps$ stands for the Hessian matrix of $u_\eps$.
By \cite{Amar:Gianni:2016C}, for every $\eps>0$, problem \eqref{eq:PDEin}--\eqref{eq:InitData} admits a unique solution
$u_\eps\in L^2\big(0,T;\XX_0(\Om)\big)\cap \CC^0\big([0,T];L^2(\Om)\cap L^2(\Memb)\big)$,
if $\overline u_0\in H^1_0(\Om)$.
\medskip

Finally, it will be useful in the sequel to define also
$\dfboth, \lfboth:Y\to\R$ as
\begin{equation*}
  \lfboth=\lfint\,\quad \text{in $\Perint$,}\qquad \lfboth=\lfout\,\quad
  \text{in $\Perout$;}
\end{equation*}
\begin{equation*}
  \dfboth=\dfint\,\quad \text{in $\Perint$,}\qquad \dfboth=\dfout\,\quad
  \text{in $\Perout$.}
\end{equation*}

\section{Homogenization of the microscopic problem}
\label{s:homog}
In the following, we will assume that the initial data satisfies
\begin{equation}\label{eq:a33}
\overline u_0\in H^1_0(\Om)\cap H^2(\Om)\,.
\end{equation}

By the trace inequality (see \cite[Proposition 1]{Amar:Gianni:2016A} and
\cite[proof of Lemma 7.1]{Amar:Andreucci:Bisegna:Gianni:2004a})
we get that $\overline u_0$ satisfies

\begin{equation}\label{eq:a79}
\eps\int_{\Memb} |\overline u_0|^2\di\sigma\le \const\,,
\qquad
\eps\int_{\Memb} |\beltramigrad \overline u_0|^2\di\sigma\le \const\,,
\end{equation}
where $\const>0$ is independent of $\eps$.
Notice that, for our purposes, it should be enough to assume that
$\overline u_0\in H^1_0(\Om)$ and satisfies \eqref{eq:a79},
but we prefer to assume \eqref{eq:a33} since it is reasonable to choose
$\overline u_0$ not depending on $\eps$.

We are interested in understanding the limiting behaviour  of the heat
potential $u_\eps$ when $\eps\to 0$; this leads us to
look at the homogenization limit of problem \eqref{eq:PDEin}--\eqref{eq:InitData}.

To this purpose, we first obtain some energy estimates for the heat potential $u_\eps$.
Multiplying \eqref{eq:PDEin} by $u_{\eps}$ and integrating, formally, by parts, we
obtain
\begin{multline}
  \label{eq:energy0}
\frac{1}{2}\int_{0}^{t}\!\!\int_{\Om} \dfbothe \pder{u^2_{\eps}}{\tau}\di x\di \tau+
  \int_{0}^{t}\!\!\int_{\Om} \lfbothe \abs{\nabla u_{\eps}}^{2} \di x\di \tau +
  \\
  \frac{\eps\alpha}{2}\int_0^t\!\!\int_{\Memb} \pder{u^2_{\eps }}{\tau}\di\sigma\di\tau +
  {\eps\beta}\int_0^t\!\!\int_{\Memb} \abs{\beltramigrad u_{\eps}}^{2}(x)\di\sigma\di\tau=0\,.
\end{multline}
Then, evaluating the time integral and taking into account the initial condition \eqref{eq:InitData},
we obtain, for all $0<t<T$,
\begin{multline}
  \label{eq:energy00}
\frac{1}{2}\int_{\Om} \dfbothe u^2_{\eps}(t)\di x+
  \int_{0}^{t}\!\!\int_{\Om} \lfbothe \abs{\nabla u_{\eps}}^{2} \di x\di \tau +
  \frac{\eps\alpha}{2}\int_{\Memb} u^2_{\eps }(t)\di\sigma +
  {\eps\beta}\int_0^t\int_{\Memb} \abs{\beltramigrad u_{\eps}}^{2}\di\sigma\di\tau=
  \\
  \frac{1}{2}\int_{\Om} \dfbothe \overline u_0^2\di x+
  \frac{\eps\alpha}{2}\int_{\Memb} \overline u^2_0\di\sigma\,.
\end{multline}
By \eqref{eq:a79} the right hand side of \eqref{eq:energy00} is stable
as $\eps\to0$, hence
\begin{multline}
  \label{eq:energy}
\sup_{t\in (0,T)}\int_{\Om} u^2_{\eps}(t)\di x+
  \int_{0}^{T}\!\!\int_{\Om} \abs{\nabla u_{\eps}}^{2} \di x\di \tau
  \\
  +\sup_{t\in (0,T)}  \eps\int_{\Memb} u^2_{\eps }(t)\di\sigma +
  \eps\int_0^T\int_{\Memb} \abs{\beltramigrad u_{\eps}}^{2}\di\sigma\di\tau\le \const\,,
\end{multline}
where $\const$ is a constant independent of $\eps$.

Notice that inequality \eqref{eq:energy} 
implies that there exists a function $u$ belonging to $ L^2\big(0,T;H^1_0(\Om)\big)$
such that, up to a subsequence, $u_\eps\rightharpoonup u$, weakly in $L^2\big(0,T;H^1_0(\Om)\big)$.
It will be our purpose to characterize the limit function $u$.

\subsection{The two-scale expansion}\label{ss:two_scale}
We summarize here, to establish the notation, some well-known
asymptotic expansions needed in the two-scale method (see, e.g.,
\cite{Bensoussan:Lions:Papanicolaou:1978}, \cite{Sanchez:1980}), when
applied to stationary or evolutive problems involving second order
partial differential equations.
Introduce the microscopic variables $y\in Y$, $y=x/\eps$ and assume
\begin{equation}
  \label{eq:uexp}
  u_{\eps}=u_{\eps}(x,y,t)=u_{0}(x,y,t) + \eps u_{1}(x,y,t) + \eps^{2}
  u_{2}(x,y,t) + \dots\,.
\end{equation}
Note that $u_{0}$, $u_{1}$, $u_{2}$ are periodic in $y$, and
$u_{1}$, $u_{2}$ are assumed to have zero integral average over
$Y$. Recalling that
\begin{equation}
  \label{eq:dermicmac}
  \Div=\frac{1}{\eps} \Div_{y} +   \Div_{x}\,, \qquad
  \nabla=\frac{1}{\eps} \nabla_{y} + \nabla_{x}\,,
\end{equation}
we compute
\begin{equation}
  \label{eq:gradexp}
  \nabla u_{\eps} =\frac{1}{\eps}\nabla_{y} u_{0} + \big(\nabla_{x}
  u_{0} +\nabla_{y} u_{1}\big) + \eps \big(\nabla_{y} u_{2} + \nabla_{x}
  u_{1}\big) +\dots\,,
\end{equation}
and
\begin{equation}
  \label{eq:laplexp}
  \Lapl u_{\eps}=\frac{1}{\eps^{2}} A_{0} u_{0} + \frac{1}{\eps}
  (A_{0}u_{1}+A_{1}u_{0}) + (A_{0}u_{2}+A_{1}u_{1}+A_{2}u_{0}) +
  \dots\,,
\end{equation}
where
\begin{equation}
  \label{eq:Axyoperators}
  A_{0}=\Lapl_{y}\,,\quad A_{1}=\Div_{y}\nabla_{x}
  +\Div_{x}\nabla_{y}\,,\quad A_{2}=\Lapl_{x}\,.
\end{equation}
Moreover, recalling \eqref{eq:beltrami1} and taking into account that the normal vector $\nu_\eps$
depends only on the microscopic variable, we obtain also
\begin{equation}
  \label{eq:dataexp}
  \beltrami u_{\eps}=\frac{1}{\eps^{2}} A^B_{0} u_{0} + \frac{1}{\eps}
  (A^B_{0}u_{1}+A^B_{1}u_{0}) + (A^B_{0}u_{2}+A^B_{1}u_{1}+A^B_{2}u_{0}) +
  \dots\,,
\end{equation}
where
\begin{multline}
  \label{eq:beltramixyoperators}
  A^B_{0}=\beltrami_{y} \,,\quad \quad A^B_{2}=\beltrami_{x}
\\
A^B_{1}=\Div^B_x\nabla^B_y+\Div^B_y\nabla^B_x
=2(Id-\nu\otimes \nu)_{ij}\partial^2_{x_iy_j}-(\Div_y \nu)\nu\cdot\nabla_x  \,.
\end{multline}

Substituting in \eqref{eq:PDEin}--\eqref{eq:InitData} the expansion \eqref{eq:uexp}, and
using \eqref{eq:dermicmac}--\eqref{eq:beltramixyoperators}, one
readily obtains, by matching corresponding powers of $\eps$, that
$u_{0}$ solves $[u_{0}]=0$ on $\Permemb$, and
\begin{equation*}
  \cpo[u_{0}]\,: \qquad
  \left\{
  \begin{aligned}
    &-\lfboth \Lapl_{y} u_{0}=0\,,\quad\text{in $\Perint$, $\Perout$;}
    \\
    & \beta\beltrami_y u_{0}
    +[\lfboth \nabla_{y} u_{0}\cdot \nu ]=0\,, \quad \text{on $\Permemb$.}
  \end{aligned}
  \right.
\end{equation*}
By the equality
\begin{align*}
0&=\int_Y \lfboth|\nabla_y u_{0}|^2\, dy+\int_{\Permemb} [\lfboth\nabla_y u_{0}\cdot\nu]u_{0}\,d\sigma
=\int_Y \lfboth|\nabla_y u_{0}|^2\,dy-\int_{\Permemb} \beta\beltrami_y u_{0} u_{0}\,d\sigma
\\
&=\int_Y \lfboth|\nabla_y u_{0}|^2\,dy+\int_{\Permemb} \beta |\nabla^B_y u_{0}|^2 \,d\sigma\,,
\end{align*}
we obtain that $u_{0}$ is independent of $y$, i.e., $u_{0}=u_{0}(x,t)$.

Moreover, $u_{1}$ satisfies
$[u_{1}]=0$ on $\Permemb$, and
\begin{equation*}
  \cpi[u_{1}]\,: \qquad
  \left\{
  \begin{aligned}
    &-\lfboth \Lapl_{y} u_{1}=0\,,\quad\text{in $\Perint$, $\Perout$;}
    \\
    & \beta\beltrami_y u_{1}+[\lfboth \nabla_{y} u_{1}\cdot \nu ]= -\beta(\Div^B_y\nabla^B_x u_{0})
    - [\lfboth \nabla_{x}
    u_{0}\cdot \nu ]\,,  \quad \text{on $\Permemb$.}
  \end{aligned}
  \right.
\end{equation*}

Following a classical approach, we introduce the factorization
\begin{equation}
  \label{eq:u0factor}
  u_{1}(x,y,t)=-\chi (y) \cdot \nabla_x u_{0}(x,t)
  =-\chi_h (y) \frac{\partial u_{0}}{\partial x_h} (x,t)\,,\qquad h=1,\dots,N\,,
\end{equation}
for a vector function $\chi:Y\to\RN$, whose components $\chi_{h}$ satisfy
\begin{alignat}2
  \label{eq:pdechi_D}
  -\lfboth \Div_{y} (\nabla_y\chi_{h}-\eh)&=0\,,&\qquad&\text{in $\Perint$,
  $\Perout$;}\\
  \label{eq:fluxcontchi_D}
   \beta\beltrami_y(\chi_{h}-y_h)&= -[\lfboth (\nabla_{y} \chi_{h} -\eh)\cdot \nu]
  \,,&\qquad &\text{on $\Permemb$;}\\
  \label{eq:circuitchi_D}
  [\chi_{h}] &=0\,,&\qquad    &\text{on $\Permemb$.}
\end{alignat}
The functions $\chi_{h}$ are also required to be periodic in
$Y$, with zero integral average on $Y$ (here, $\eh$ denotes the $h$ vector of the canonical basis of $\R^N$).
We note that \cite{Amar:Gianni:2016C} assures existence and uniqueness
of the cell functions $\chi_h\in \CC^\infty_\#(Y)$, for $h=1,\dots,N$
(here and in the following, the subscript $\#$ denotes the $Y$-periodicity).

Finally, $u_{2}$ solves $[u_{2}]=0$ on $\Permemb$, and
\begin{equation*}
  \cpii[u_{2}]\,: \qquad
  \left\{
  \begin{aligned}
    &-\lfboth \Lapl_{y} u_{2}=-\dfboth u_{0t}+\lfboth \Lapl_{x} u_{0} + 2\lfboth
    \frac{\partial^{2} u_{1}}{\partial x_{j}\partial y_{j}}\,,
    \quad\text{in $\Perint$, $\Perout$;}
    \\
    & \beta \beltrami_y u_{2}+[\lfboth \nabla_{y} u_{2}\cdot \nu ]=
    \\
    &\quad \alpha u_{0t}-\beta\beltrami_x u_{0}
    -\beta\Div^B_x\nabla^B_yu_{1}-\beta\Div^B_y\nabla^B_xu_{1}- [\lfboth \nabla_{x}
    u_{1}\cdot \nu ]\,, \quad \text{on $\Permemb$.}
  \end{aligned}
  \right.
\end{equation*}

The limiting equation for $u_{0}$ is finally obtained as a compatibility
condition for $\cpii[u_{2}]$, and amounts to
\begin{equation}
\label{eq:compatibile}
\begin{aligned}
& \int_Y \Big(-\dfboth u_{0t}+ \lfboth\Delta_xu_0 +2\lfboth\frac{\partial^2u_1}{\partial x_j\partial y_j}\Big)\di y
=\int_\Permemb[\lfboth\nabla_yu_2\cdot\nu]\di \sigma=
\\
& \int_\Permemb \Big(\alpha u_{0t}-[\lfboth\nabla_xu_1\cdot\nu]-\beta\beltrami_yu_2
 -\beta\beltrami_xu_0-\beta\Div^B_x\beltramigrad_y u_1
- \beta\Div^B_y\beltramigrad_x u_1\Big)\di \sigma\,.
\end{aligned}
\end{equation}
We replace now the factorization \eqref{eq:u0factor} in the previous equality and we take
into account that
\begin{alignat}2
\label{eq:a7} & 2\int_Y \lfboth\frac{\partial^2u_1}{\partial x_j\partial y_j}\di y= -2\int_\Permemb [\lfboth\nabla_xu_1\cdot\nu]\di\sigma\,,
\\
\label{eq:a8} & -\int_\Permemb[\lfboth\nabla_xu_1\cdot\nu]\di\sigma =
\Div\left(\Big(\int_\Permemb [\lfboth](\nu\otimes\chi)\di\sigma  \Big)\nabla u_0\right)\,,
\\
\label{eq:a9} & -\int_\Permemb \beta\beltrami_y u_2\di\sigma =0\,,
\\
\label{eq:a10} & -\int_\Permemb \beta\beltrami_x u_0\di\sigma=-\beta|\Permemb|\Delta u_0 +
\Div\left(\Big(\int_\Permemb \beta(\nu\otimes\nu)\di\sigma\Big)\nabla u_0\right)\,,
\end{alignat}
 \begin{alignat}2
\label{eq:a11} &
-\int_\Permemb \beta\Div^B_x\beltramigrad_y u_1\di\sigma=
\Div\left(\Big(\int_\Permemb \beta(I-\nu\otimes\nu)\nabla_y\chi\di\sigma\Big)\nabla u_0\right)\,,
\\
\label{eq:a12} &
{\begin{aligned}
& -\int_\Permemb \beta\Div^B_y\beltramigrad_x u_1\di\sigma=0\,,
\end{aligned}}
\end{alignat}
where \eqref{eq:a12} follows from \eqref{eq:a66}, since $\Permemb$ has no boundary.
Hence, we obtain for the homogenized solution $u_0$ the parabolic equation
\begin{equation}
  \label{eq:limitPDE_kpinf}
\gammafinale u_{0t}  -\Div \Big( (\lfav I  + \dfmat) \nabla u_{0} \Big)
  =0\,, \qquad \text{in $\Om_T$,}
\end{equation}
where
\begin{multline}
  \label{eq:dfmatD_dfav}
  \gammafinale=\dfint \abs{\Perint}+\dfout\abs{\Perout}+ \alpha |{\Permemb}|\,,
  \qquad
  \lfav=\lfint\abs{\Perint}+\lfout\abs{\Perout}\,,
  \\
  \dfmat
   =\int_{\Permemb}[\lfboth] (\nu\otimes\chi)\di  \sigma
  +\beta\int_{\Permemb}\Big((I-\nu\otimes \nu)+(\nu\otimes\nu)\nabla_y\chi-\nabla_y\chi\Big)
  \di  \sigma=
\\
=\int_{\Permemb}[\lfboth] (\nu\otimes\chi)\di  \sigma
  -\beta\int_{\Permemb}\beltramigrad_y(\chi-y)  \di  \sigma
  \,.
\end{multline}
Clearly, equation \eqref{eq:limitPDE_kpinf} must be complemented with a boundary
and an initial condition which are $u_0=0$ on $\partial \Om\times (0,T) $
and $u_0(x,0)=\overline u_0(x)$ in
$\Om$, respectively, as follows from the microscopic problem \eqref{eq:PDEin}--\eqref{eq:InitData}.
Indeed, by \eqref{eq:energy} we obtain that $\{u_\eps\}$ converges weakly in $L^2\big(0,T;H^1_0(\Om)\big)$,
which implies the weak convergence of the trace on $\partial\Om$, while the initial data is already
included in the weak formulation of the problem.
\medskip

\begin{thm}\label{l:l1}
The matrix $\lfav I  + \dfmat$ is symmetric and positive definite.
\end{thm}

\begin{proof}
We first prove the symmetry. By \eqref{eq:a5bis}, we have
\begin{equation}\label{eq:a5}
-\int_{\Permemb} \beltramigrad_y y_h\cdot\beltramigrad_y\chi_j\di\sigma=
-\int_{\Permemb} (\eh-\nu_h\nu)\cdot\beltramigrad_y\chi_j\di\sigma
=-\int_{\Permemb}(\beltramigrad_y\chi_j)_h\di\sigma\,;
\end{equation}
then, taking into account \eqref{eq:pdechi_D}--\eqref{eq:circuitchi_D}, we obtain
\begin{multline}\label{eq:a6}
0=-\int_Y \lfboth\Delta_y(\chi_h-y_h)\,\chi_j\di y
=\int_Y\lfboth\nabla_y(\chi_h-y_h)\cdot\nabla_y\chi_j \di y
-\beta\int_{\Permemb}\beltrami_y(\chi_h-y_h)\,\chi_j\di\sigma
\\
=\int_Y\lfboth\nabla_y\chi_h\cdot\nabla_y\chi_j \di y
-\int_Y\lfboth\eh\cdot\nabla_y\chi_j\di y
+\beta\int_{\Permemb}\beltramigrad_y(\chi_h-y_h)\cdot\beltramigrad_y\chi_j\di\sigma
\\
=\int_Y\lfboth\nabla_y\chi_h\cdot\nabla_y\chi_j \di y
+\int_{\Permemb}[\lfboth]\nu_h\,\chi_j\di\sigma
+\beta\int_{\Permemb}\nabla^B_y\chi_h\nabla^B_y\chi_j\di\sigma
-\beta\int_{\Permemb}\nabla^B_y y_h\nabla^B_y\chi_j\di\sigma
\\
=\int_Y\lfboth\nabla_y\chi_h\cdot\nabla_y\chi_j \di y
+\int_{\Permemb}[\lfboth]\nu_h\,\chi_j\di\sigma
+\beta\int_{\Permemb}\nabla^B_y\chi_h\nabla^B_y\chi_j\di\sigma
-\beta\int_{\Permemb}(\beltramigrad_y\chi_j)_h\di\sigma\,.
\end{multline}
From \eqref{eq:dfmatD_dfav} and \eqref{eq:a6}, we can rewrite
$$
A^{hom}=\int_{\Permemb}\beta(I-\nu\otimes\nu)\di\sigma-\int_Y\lfboth(\nabla_y\chi\otimes\nabla_y\chi) \di y
-\int_{\Permemb}\beta(\nabla^B_y\chi\otimes\nabla^B_y\chi)\di\sigma\,,
$$
which gives the symmetry of the matrix $A^{hom}$ and hence the symmetry of the whole matrix $\lfav I  + \dfmat$.

Let us now prove that it is also positive definite.
Firstly, we observe that, using \eqref{eq:a5} and \eqref{eq:a6}, we obtain
\begin{equation*}
\begin{aligned}
& \hphantom{=}\int_Y \lfboth\nabla(\chi_h-y_h)\cdot\nabla(\chi_j-y_j)\di y
+\beta\int_{\Permemb}\nabla^B_y(\chi_h-y_h)\nabla^B_y(\chi_j-y_j)\di\sigma
\\
& =
\int_Y \lfboth \nabla\chi_h\cdot\nabla\chi_j\di y+\int_Y \lfboth\eh\cdot\ej\di y
-\int_Y \lfboth\nabla\chi_h\cdot\ej\di y
-\int_Y \lfboth\nabla\chi_j\cdot\eh\di y
\\
&\qquad +\beta\int_\Permemb \nabla^B\chi_h\cdot\nabla^B\chi_j\di \sigma
+\beta\int_\Permemb \nabla^B y_h\cdot\nabla^B y_j\di \sigma
\\
&\qquad -\beta\int_\Permemb \nabla^B\chi_h\cdot\nabla^B y_j\di \sigma
 -\beta\int_\Permemb \nabla^B\chi_j\cdot\nabla^B y_h\di \sigma
\\
& = \int_Y \lfboth \nabla\chi_h\cdot\nabla\chi_j\di y+\int_Y\lfboth\delta_{hj}\di y
+\int_\Permemb [\lfboth]\chi_h\nu_j\di\sigma
+\int_\Permemb [\lfboth]\chi_j\nu_h\di\sigma+
\\
&\qquad +\beta\int_\Permemb \nabla^B\chi_h\cdot\nabla^B\chi_j\di \sigma
+\beta\int_\Permemb \nabla^B y_h\cdot\nabla^B y_j\di \sigma
-\beta\int_\Permemb (\beltramigrad_y\chi_h)_j\di \sigma
-\beta\int_\Permemb (\beltramigrad_y\chi_j)_h\di \sigma
\\
& = \int_Y \lfboth \nabla\chi_h\cdot\nabla\chi_j\di y
+\int_Y\lfboth\delta_{hj}\di y
+\beta\int_\Permemb \nabla^B\chi_h\cdot\nabla^B\chi_j\di \sigma
+\beta\int_\Permemb \nabla^B y_h\cdot\nabla^B y_j\di \sigma
\\
& \quad-2\int_Y \lfboth \nabla\chi_h\cdot\nabla\chi_j\di y-2\beta\int_\Permemb\nabla^B\chi_h\nabla^B\chi_j\di\sigma
\\
& = \int_Y\lfboth\delta_{hj}\di y-\int_Y \lfboth \nabla\chi_h\cdot\nabla\chi_j\di y
+\beta\int_\Permemb(\delta_{hj}-\nu_h\nu_j)\di\sigma
-\beta\int_\Permemb\nabla^B\chi_h\nabla^B\chi_j\di\sigma\,.
\end{aligned}
\end{equation*}
Then, we can rewrite
\begin{equation*}
\begin{aligned}
&\hphantom{=}
(\lfav I  + \dfmat)_{hj}  = \int_Y \lfboth \delta_{hj}\di y + \int_\Permemb \beta\delta_{hj}\di\sigma
-\int_\Permemb \beta\nu_h\nu_j\di\sigma
\\
& \qquad -\int_Y\lfboth\nabla\chi_h\cdot\nabla\chi_j\di y
-\int_\Permemb \beta \nabla^B\chi_h\cdot\nabla^B\chi_j\di\sigma
\\
& = \int_Y \lfboth\nabla(\chi_h-y_h)\cdot\nabla(\chi_j-y_j)\di y
+\int_\Permemb \beta\nabla^B(\chi_h-y_h)\cdot\nabla^B(\chi_j-y_j)\di \sigma\,.
\end{aligned}
\end{equation*}
Finally, setting $\lfboth_{min} =\min(\lfint,\lfout)$ and
using Jensen's inequality, we obtain
\begin{equation*}
\begin{aligned}
&\hphantom{=}
\sum_{h,j=1}^{N}(\lfav I  + \dfmat)_{hj} \xi_h\xi_j =
\int_Y \sum_{h,j=1}^{N}\lfboth(\nabla\chi_h\xi_h-\eh\xi_h)\cdot(\nabla\chi_j\xi_j-\ej\xi_j)\di y
\\
&\qquad+\int_\Permemb \sum_{h,j=1}^{N}\beta
\nabla^B(\chi_h\xi_h-y_h\xi_h)\cdot\nabla^B(\chi_j\xi_j-y_j\xi_j)\di \sigma
\\
& \geq \lfboth_{min} \int_Y \big|\sum_{h=1}^{N}(\nabla\chi_h\xi_h-\eh\xi_h)\Big|^2\di y
+\beta\int_\Permemb \big|\sum_{h=1}^{N}\nabla^B(\chi_h\xi_h-y_h\xi_h)\big|^2\di \sigma
\\
& \geq \lfboth_{min}\left|\int_Y \sum_{h=1}^{N}(\nabla\chi_h\xi_h-\eh\xi_h)\di y\right|^2
+\beta|\Permemb|\left|\frac{1}{|\Permemb|}
\int_\Permemb \sum_{h=1}^{N}\nabla^B(\chi_h\xi_h-y_h\xi_h)\di \sigma\right|^2
\\
& \geq \lfboth_{min}\sum_{j=1}^{N}\left( \sum_{h=1}^{N}(\xi_h\int_Y
\frac{\partial\chi_h}{\partial y_j}\di y-\delta_{hj}\xi_h)\right)^2
+\frac{\beta}{|\Permemb|}\left|\sum_{h=1}^{N}\int_\Permemb \nabla^B(\chi_h\xi_h-y_h\xi_h)\di \sigma\right|^2
\\
& \geq \lfboth_{min}\sum_{j=1}^{N}\left( \sum_{h=1}^{N}\xi_h\int_{\partial Y}
\chi_h\,n_j\di \sigma-\xi_j\right)^2
= \lfboth_{min}|\xi|^2
\end{aligned}
\end{equation*}
where we have denoted by $n=(n_1,\dots,n_N)$ the outward unit normal to $\partial Y$.
Moreover, we
remark that the last integral vanishes because of the periodicity of the cell
function $\chi_h$.

This proves that the homogenized matrix is positive definite and concludes the theorem.
\end{proof}

\begin{Remark}
\label{r:r1}
We note that the homogenized matrix is positive definite independently of the strict positivity of
$\beta$.
\end{Remark}
\medskip

Once proved Theorem \ref{l:l1}, the existence of a unique solution for
equation \eqref{eq:limitPDE_kpinf} complemented
with suitable initial and boundary conditions is standard. The next proposition state the regularity of this solution,
which is a property needed in order to obtain the error estimate.

\begin{prop}\label{p:p3}
Assume that $\overline u_0\in \CC^\infty_c(\Om)$ (i.e. $\overline u_0$ has compact support in $\Om$).
Then, the solution $u_0$ to equation \eqref{eq:limitPDE_kpinf} satisfying the
homogeneous boundary condition on $\partial\Om\times[0,T]$
and the initial condition $u(x,0)=\overline u_0(x)$ in $\Om$ belongs to $\CC^\infty(\overline\Om\times[0,T])$.
\end{prop}

\begin{proof}
The result can be obtained applying \cite[Theorem 12 in Section 5]{Friedman:1999}.
\end{proof}

\begin{Remark}\label{r:2}
Actually, the asserted $\CC^\infty$-regularity of the homogenized solution $u_0$ is far from being optimal
in order to obtain the error estimate proved in Section \ref{s:error}. Indeed, to this purpose, it
is enough to have that $u_0\in \CC^0\big([0,T];\CC^3(\overline\Om)\big)$
and this is guaranteed if, for instance
$\overline u_0\in \CC^4(\overline\Om)$ and satisfies the compatibility conditions
\begin{equation}\label{eq:compatib1}
L_{hom}\overline u_0(x)=0\,,\quad\text{and}\quad L_{hom}^2\overline u_{0} (x):=
L_{hom}\big(L_{hom}\overline u_0(x))=0\,,
\quad \text{on $\partial\Om$,}
\end{equation}
where $L_{hom}=-\Div \big((\lfav I  + \dfmat) \nabla \big)$, with $\lfav$ and $\dfmat$ defined in
\eqref{eq:dfmatD_dfav}.
However, we prefer the simpler assumptions of Proposition \ref{p:p3}, since we
are not interested in stating which are the minimal conditions
to be satisfied by the initial data in order to obtain the optimal regularity of the homogenized solution.
\end{Remark}

For further use
(taking into account the system satisfied by $u_2$ and \eqref{eq:limitPDE_kpinf}),
we introduce the factorization of the function $u_2$ in terms of the homogenized
solution $u_0$; i.e.,
\begin{equation}\label{eq:a1}
u_2(x,y,t)=\widetilde\chi_{ij}(y) \frac{\partial^2u_0}{\partial x_ix_j}(x,t)\,,
\qquad i,j=1,\dots,N\,,
\end{equation}
where the functions $\widetilde\chi_{ij}:Y\to\R$ satisfy
\begin{alignat}2
  \label{eq:pdechi_ij}
    -\lfboth \Delta_y\widetilde\chi_{ij} =-\frac{\dfboth}{\gammafinale}(\lfav\delta_{ij}&+a^{hom}_{ij})
  + \lfboth\delta_{ij}  -2\lfboth\frac{\partial \chi_i}{\partial y_j}=:F
  \,,&\qquad&\!\!\!\!\!\!\!\!\!\!\!\!\!\!\!\!\!\text{in $\Perint$, $\Perout$;}\\
  \label{eq:fluxcontchi_ij}
   \beta\beltrami_y\widetilde\chi_{ij}+[\lfboth\nabla_y\widetilde \chi_{ij}\cdot\nu]&=
   {\begin{aligned}
   & \\
&   \frac{\alpha}{\gammafinale}(\lfav\delta_{ij}+a^{hom}_{ij})-
\beta\big(\delta_{ij}-(\nu\otimes\nu)_{ij}\big)
   \\
\!\!\!\!\!\!\!\!\!\!\!\!\!\!\!\!\!\!\!\!\!\!\!\!\!\!\!\!\!\!\!\!\!\!\!\!\!\!\!\!\!\!\!\!
+2\beta \big(I-(\nu\otimes\nu)\big)_i &\cdot\nabla\chi_j-\beta\nu_j\chi_i\Div\nu+[\lfboth\nu_i]\chi_j=:G
  \,,
  \end{aligned}}
  &\qquad &\text{on $\Permemb$;}\\
  \label{eq:circuitchi_ij}
  [\widetilde  \chi_{ij}] &=0\,,&\qquad    &\text{on $\Permemb$.}
\end{alignat}
The functions $\widetilde\chi_{ij}$ are also required to be periodic in
$Y$, with zero integral average on $Y$.
In order to obtain \eqref{eq:pdechi_ij}--\eqref{eq:circuitchi_ij} we have taken into account
\eqref{eq:beltramixyoperators}, which gives
$$
\Div^B_x(\nabla_y^B \phi)+\Div^B_y(\nabla_x^B \phi) =
2(\delta_{ij}-\nu_i\nu_j)\frac{\partial^2\phi}{\partial x_i\partial y_j}-\nu_j\frac{\partial\nu_i}{\partial y_i}
\frac{\partial\phi}{\partial x_j}\,,
$$
with $\phi(x,y,t)=u_1(x,y,t)=-\chi(y)\cdot\nabla_xu_0(x,t)$ and the usual summation convention for repeated indexes.
By \cite{Amar:Gianni:2016C}, problem \eqref{eq:pdechi_ij}--\eqref{eq:circuitchi_ij} admits
a unique solution $\widetilde\chi_{ij}\in\CC^\infty_\#(Y)$, for $i,j=1,\dots,N$, since it is easy to
check that
$$
\int_Y F\di y=\int_{\Permemb} G\di\sigma\,.
$$

\section{Error estimate}\label{s:error}
In this section we prove that the limit $u$ of the sequence $\{u_\eps\}$ of the solutions of problem \eqref{eq:PDEin}--\eqref{eq:InitData} coincides with the solution $u_0$ of equation \eqref{eq:limitPDE_kpinf}.
In order to achieve this result, we will state an error estimate for the sequence $\{u_\eps\}$, which gives
the rate of convergence of such a sequence to the homogenized function $u_0$, in a suitable norm, thus obtaining a stronger convergence result with respect to the one obtained in our previous paper \cite{Amar:Gianni:2016A}.
However, this result needs extra-regularity assumptions on the initial data
$\overline u_0(x)$ (see Proposition \ref{p:p3} and Remark \ref{r:2}), which assure more regularity of the homogenized solution $u_0$.

\begin{thm}\label{t:t1}
Assume that $\overline u_0\in \CC^\infty_c(\Om)$.
Let $u_{0}$ be the smooth solution of \eqref{eq:limitPDE_kpinf}, satisfying
the initial condition $u_0(x,0)=\overline u_0(x)$ in $\Om$ and
the boundary condition $u_0(x,t)=0$ on $\partial\Om\times(0,T)$; moreover,
let $u_1$ be the function defined in \eqref{eq:u0factor}.
Then
\begin{alignat}2
\label{eq:error_na}
&  \Vert u_{\eps}-(u_{0}+\eps u_1)\Vert_{L^2(0,T;H^1(\Om))} &\le \const \sqrt\eps\,,
\\
\label{eq:error_nb}
&  \Vert u_{\eps}-u_{0}\Vert_{L^2(\Om_T)} &\le \const \sqrt\eps\,,
 \end{alignat}
for a proper constant $\const>0$, independent of $\eps$.
\end{thm}

\begin{proof}
Let us define the rest function
\begin{equation*}
  \reps(x,t)=\big(u_{\eps}(x,t)-u_{0}(x,t)-\eps u_{1}(x,x/\eps,t)\big)\eps^{-1}\,,
  \qquad x\in\Om\,, t>0\,.
\end{equation*}
Separately in $\Omint$ and in $\Omout$, we get
\begin{align*}
\dfbothe\frac{\partial \reps}{\partial t}  -\Div (\lfbothe \nabla \reps)&=
\frac{1}{\eps}\Big\{ -\dfbothe\frac{\partial u_0}{\partial t} + \Div (\lfbothe\nabla u_{0})
-\dfbothe\eps\frac{\partial u_1}{\partial t} + \eps \Div( \lfbothe \nabla u_{1})
  \Big\}
\\
  &= \frac{1}{\eps} \Big\{-\dfbothe\frac{\partial u_0}{\partial t} +
  \lfbothe \Lapl_{x}u_{0} + 2 \lfbothe u_{1x_{h}y_{h}} \Big\}
  -\dfbothe\frac{\partial u_1}{\partial t} +  \lfbothe \Lapl_{x}u_{1}+ \frac{1}{\eps^2} \lfbothe
  \Lapl_{y}u_{1}
  \\
  &=    -\frac{1}{\eps} \lfbothe \Lapl_{y} u_{2} -\dfbothe\frac{\partial u_1}{\partial t} + \lfbothe \Lapl_{x}
  u_{1}\qefed \oldfeps-\dfbothe\frac{\partial u_1}{\partial t} \,.
\end{align*}
Moreover,
\begin{equation*}
[\reps]=0\,,\qquad \reps(x,0)=-u_1(x,x/\eps,0)=\chi(x/\eps)\cdot\nabla_xu_0(x,0)
=\chi(x/\eps)\cdot\nabla_x\overline u_0(x,0)\,,
\end{equation*}
and
\begin{align*}
& \eps\alpha \frac{\partial\reps}{\partial t}-\eps\beta\beltrami\reps
= \frac{1}{\eps}\left\{\eps\alpha \frac{\partial u_\eps}{\partial t}-\eps\beta\beltrami u_\eps
-\eps\alpha \frac{\partial u_0}{\partial t}+\eps\beta\beltrami u_0\right\}
\\
&\qquad -\left\{\eps\alpha \frac{\partial u_1}{\partial t}-\eps\beta\beltrami u_1\right\}
\\
& = \frac{1}{\eps}[\lfbothe\nabla u_\eps\cdot\nu_\eps]-\alpha \frac{\partial u_0}{\partial t}
+\beta\beltrami_x u_0+\beta \Div^B_x\nabla^B_y u_1+\beta \Div^B_y\nabla^B_x u_1
 \\
 &\qquad  -\eps\alpha\frac{\partial u_1}{\partial t}+\eps\beta\beltrami_x u_1
+ \frac{1}{\eps}\left(\beta\beltrami_yu_1+\beta\Div^B_y\nabla^B_x u_0+\beta\Div^B_x\nabla^B_y u_0\right)
+\frac{1}{\eps^2}\beta\beltrami_y u_0
\\
 &  =  \frac{1}{\eps}[\lfbothe\nabla u_\eps\cdot\nu_\eps]-[\lfbothe(\nabla_xu_1+\nabla_yu_2)\cdot\nu_\eps]
 -\beta\beltrami_yu_2
 \\
 &\qquad -\eps(\alpha\frac{\partial u_1}{\partial t}-\beta\beltrami_{x}u_1)
 -\frac{1}{\eps}[\lfbothe(\nabla_xu_0+\nabla_y u_1)\cdot\nu_\eps]
 \\
 & =[\lfbothe\nabla \reps\cdot\nu_\eps]-\eps(\alpha\frac{\partial u_1}{\partial t}-\beta\beltrami_{x}u_1)-[\lfbothe\nabla_yu_2\cdot\nu_\eps]-\beta\beltrami_yu_2 \,,
  \end{align*}
where we have taken into account the problems satisfied by $u_1$ and $u_2$
($u_{1}$ and $u_{2}$ are defined in Subsection \ref{ss:two_scale})
and the fact that $\Div^B_x\nabla^B_y u_0=0$ and $\beltrami_y u_0=0$.

Let us now introduce the corrected rest function
\begin{equation*}
  \rteps=\reps + u_{1}\phi_{\eps}\,,
\end{equation*}
where $\phi_{\eps}$ is a cut-off function equal to $1$ in a
neighbourhood of $\partial\Om$, and such that
\begin{equation*}
  \phi_{\eps}(x)=0 \quad\text{if}\quad \dist(x,\partial\Om)\ge
  \gamma_{0}\eps\,.
\end{equation*}
Clearly, $\phi_{\eps}\equiv 0$ on $\Memb$ (since $\dist(\Memb,\partial\Om)\geq \const_0\eps$, by the
assumptions made in Subsection \ref{ss:geometric}), so that $\reps=\rteps$ on $\Memb$.
We may assume $0\le\phi_{\eps}\le1$,
$\abs{\nabla\phi_{\eps}}\le\gamma/\eps$.
The function $\rteps$ satisfies $[\rteps]=0$ on $\Memb$ and
\begin{alignat}2
  \label{eq:restPDE}
  \dfbothe\frac{\partial\rteps}{\partial t}-\lfbothe \Lapl\rteps&=\oldfeps
  -\dfbothe\frac{\partial u_1}{\partial t} +\dfbothe\phi_{\eps}\frac{\partial u_1}{\partial t}
  -\lfbothe\Lapl(u_{1}\phi_{\eps})\,,&
  \qquad & \text{in $\Omint$, $\Omout$;}\\
  \label{eq:restinitdata}
  \rteps(x,0)&=\chi(x/\eps)\cdot\nabla_x\overline u_0(x,0)(1-\phi_\eps) \,,&\qquad & \text{on $\Om$;}\\
  \label{eq:restDir}
  \rteps&=0\,, &\qquad &\text{on $\partial\Om$,}
\end{alignat}
and on $\Memb$
\begin{equation}\label{eq:restfluxcont}
  \begin{aligned}
  \eps\alpha \frac{\partial\rteps}{\partial t}-\eps\beta\beltrami\rteps
  =    & [\lfbothe\nabla \reps\cdot\nu_\eps]-\eps(\alpha\frac{\partial u_1}{\partial t}-\beta\beltrami_{x}u_1)-[\lfbothe\nabla_yu_2\cdot\nu_\eps]-\beta\beltrami_yu_2
   \\
  =  & [\lfbothe\nabla \rteps\cdot\nu_\eps]-\eps(\alpha\frac{\partial u_1}{\partial t}-\beta\beltrami_{x}u_1)-[\lfbothe\nabla_yu_2\cdot\nu_\eps]-\beta\beltrami_yu_2   \,.
    \end{aligned}
  \end{equation}
Note that the correction $u_{1}\phi_{\eps}$ has been introduced
precisely in order to guarantee \eqref{eq:restDir}.
Multiply \eqref{eq:restPDE} by $\rteps$ and integrate by parts;
by virtue of \eqref{eq:restDir}, we get
\begin{multline}
  \label{eq:restenergy_i}
\intotom \big\{\oldfeps  - \lfbothe \Lapl(u_{1}\phi_{\eps})\big\}
  \rteps\di x\di\tau
  -\intotom \big\{\dfbothe\frac{\partial u_1}{\partial  \tau}
  (1-\phi_{\eps})\big\}\rteps \di x\di\tau=
  \\
\frac{1}{2}\intotom\dfbothe\frac{\partial \rteps^2}{\partial  \tau}\di x\di  \tau
+  \intotom \lfbothe \abs{\nabla \rteps}^{2}\di x\di\tau +
  \intotme [\lfbothe \nabla\rteps\cdot\nu_\eps]
  \rteps \di\sigma\di\tau =
\\
 \frac{1}{2}\int_{\Om}\dfbothe\rteps^2(x,t)\di x-\frac{1}{2}\int_{\Om}\dfbothe\rteps^2(x,0)\di x
+  \intotom \lfbothe \abs{\nabla \rteps}^{2}\di x\di\tau
\\
+\frac{\eps}{2}  \int_{\Memb}\alpha \rteps^2(x,t)\di\sigma
- \frac{\eps}{2}  \int_{\Memb}\alpha \rteps^2(x,0)\di\sigma
+\eps\beta \intotme|\nabla^B  \rteps|^2 \di\sigma\di\tau
\\
+\eps \intotme(\alpha\frac{\partial u_1}{\partial t}-\beta\beltrami_xu_1)\rteps\di\sigma\di\tau
 +\intotme(\beta\beltrami_yu_2+[\lfboth\nabla_yu_2\cdot\nu_\eps])\rteps\di\sigma\di\tau\,.
\end{multline}
This implies
\begin{multline*}
\frac{1}{2}\int_{\Om}\dfbothe\rteps^2(x,t)\di x
+\frac{\eps}{2}  \int_{\Memb}\alpha \rteps^2(x,t)\di\sigma
+  \intotom \lfbothe \abs{\nabla \rteps}^{2}\di x\di\tau
+  \eps\beta \intotme|\nabla^B  \rteps|^2 \di\sigma\di\tau
=
\\
\frac{1}{2}\int_{\Om}\dfbothe\rteps^2(x,0)\di x+
\frac{\eps}{2}  \int_{\Memb}\alpha \rteps^2(x,0)\di\sigma
-\eps \intotme(\alpha\frac{\partial u_1}{\partial  \tau}-
\beta\beltrami_xu_1)\rteps\di\sigma\di\tau
\\
 -\intotme(\beta\beltrami_yu_2+[\lfbothe\nabla_yu_2\cdot\nu_\eps])\rteps\di\sigma\di\tau
 \\
+ \intotom \big\{\oldfeps  - \lfbothe \Lapl(u_{1}\phi_{\eps})\big\}
  \rteps\di x\di\tau
  -\intotom \big\{\dfbothe\frac{\partial u_1}{\partial  \tau} (1-\phi_{\eps})\big\}\rteps \di x\di\tau
\end{multline*}
Next, compute
\begin{multline}
  \label{eq:rest_feps}
  \intotom \oldfeps   \rteps\di x\di\tau = \intotom
  \lfbothe\big\{-\frac{1}{\eps}\Lapl_{y}u_{2} +  \Lapl_{x}u_{1}\big\}
  \rteps \di x\di\tau\\
  = \intotom \lfbothe\big\{-\frac{1}{\eps}\Lapl_{y}u_{2} - \Div_{x}
  (\nabla_{y} u_{2})\big\} \rteps \di x\di\tau +
  \intotom \lfbothe\big\{ \Div_{x} (\nabla_{y} u_{2}) +
  \Lapl_{x}u_{1}\big\}\rteps\di x\di\tau\\
  =-\intotom \Div (\lfbothe \nabla_{y}u_{2}) \rteps\di x\di\tau +
  \intotom \big\{\lfbothe \Div_{x} (\nabla_{y} u_{2}) +\lfbothe
  \Lapl_{x}u_{1}\big\}\rteps\di x\di\tau\\
  =  \intotme [\lfbothe
  \nabla_{y}u_{2} \cdot\nu_\eps ]\rteps\di\sigma\di\tau +
  \intotom \lfbothe \nabla_{y}u_{2}\cdot \nabla \rteps\di x\di\tau \\
  + \intotom \big\{\lfbothe \Div_{x} (\nabla_{y} u_{2}) +\lfbothe
  \Lapl_{x}u_{1}\big\}\rteps\di x\di\tau
\end{multline}
Note that the last integral in \eqref{eq:rest_feps} can be bounded in the following way
\begin{equation*}
  \intotom \big\{\lfbothe \Div_{x} (\nabla_{y} u_{2}) +\lfbothe
  \Lapl_{x}u_{1}\big\}\rteps\di x\di\tau \le \const(\delta) + \delta \intotom
  \rteps^{2}\di x\di\tau\,,
\end{equation*}
where $\delta>0$ will be chosen in the following. We exploit here the
estimate
\begin{equation}
  \label{eq:regest_j}
 \intotom (u_{2x_{i}y_{i}}^{2} + u_{1x_{i}x_{i}}^{2})\di x\di \tau \le
 \const\,,
\end{equation}
which is a consequence of the regularity of the cell functions $\chi$ and $\widetilde \chi$
(recall \eqref{eq:u0factor}--\eqref{eq:circuitchi_D} and \eqref{eq:a1}--\eqref{eq:circuitchi_ij})
and of the homogenized function $u_0$.
Similarly, for $\delta'=\min(\lfint,\lfout)/2$,
\begin{multline}
  \label{eq:lapl_au1phieps}
  -\intotom \lfbothe \Lapl(u_{1}\phi_{\eps}) \rteps\di x\di\tau = \intotom
  \lfbothe\nabla (u_{1}\phi_{\eps})\cdot \nabla \rteps \di x\di\tau \le \delta'
  \intotom \abs{\nabla\rteps}^{2}\di x\di\tau \\
  + \frac{\const(\delta')}{\eps^{2}} \abs{\{x\in\Om\mid \dist(x,\partial\Om)\le
  \gamma_{0}\eps\}} \le
  \delta' \intotom \abs{\nabla\rteps}^{2}\di
  x\di\tau +\frac{\const(\delta')}{\eps}\,,
\end{multline}
where, again due to the stated regularity of $\chi$ and $u_0$, we used
\begin{equation}
  \label{eq:regest_jj}
  \sup\limits_{x\in\Om\,,\,y\in Y\,,\,0<t<T} \big\{\abs{u_{1}}
  +\abs{\nabla_{x}u_{1}} +\abs{\nabla_{y}u_{1}}\big\} (x,y,t) <+\infty\,.
\end{equation}
Moreover, for $\delta^{\prime\prime}$ which will be chosen later, we obtain
\begin{multline*}
\intotme(\beta\beltrami_yu_2)\rteps\di\sigma\di\tau=
\eps\beta\intotme(\frac{1}{\eps}\Div^B_y\nabla_y^Bu_2+\Div^B_x\nabla_y^Bu_2)\rteps\di\sigma\di\tau
\\
-\eps\beta\intotme(\Div^B_x\nabla_y^Bu_2)\rteps\di\sigma\di\tau=
\\
-\eps\beta\intotme\nabla_y^Bu_2\nabla^B\rteps\di\sigma\di\tau
-\eps\beta\intotme(\Div^B_x\nabla_y^Bu_2)\rteps\di\sigma\di\tau=
\\
\const(\delta^{\prime\prime})+\delta^{\prime\prime}\eps\intotme|\nabla^B\rteps|^2\di\sigma\di\tau
+\const(\delta^{\prime\prime})+\delta^{\prime\prime}\eps\intotme\rteps^2\di\sigma\di\tau\,.
\end{multline*}
Here, we use
\begin{equation*}
\eps\intotme (|\nabla_y^Bu_{2}|^{2} + |\Div^B_x\nabla_y^Bu_2|^{2})\di \sigma\di \tau \le \const\,,
\end{equation*}
which is again a consequence of the regularity of $\widetilde\chi$ and $u_0$.

Combining the previous estimates, we have
\begin{multline}
  \label{eq:restenergy_ii}
  \frac{1}{2}\int_{\Om}\dfbothe\rteps^2(x,t)\di x
+\frac{\eps}{2}  \int_{\Memb}\alpha \rteps^2(x,t)\di\sigma
+  \intotom \lfbothe \abs{\nabla \rteps}^{2}\di x\di\tau
+  \eps\beta \intotme|\nabla^B  \rteps|^2 \di\sigma\di\tau
\le
\\
\frac{1}{2}\int_{\Om}\dfbothe\rteps^2(x,0)\di x+
\frac{\eps}{2}  \int_{\Memb}\alpha \rteps^2(x,0)\di\sigma
-\eps \intotme(\alpha\frac{\partial u_1}{\partial  \tau}-
\beta\beltrami_xu_1)\rteps\di\sigma\di\tau
\\
+\const(\delta^{\prime\prime})+\delta^{\prime\prime}\eps\intotme|\nabla^B\rteps|^2\di\sigma\di\tau
+\delta^{\prime\prime}\eps\intotme\rteps^2\di\sigma\di\tau
 -\intotme[\lfbothe\nabla_yu_2\cdot\nu_\eps]\rteps\di\sigma\di\tau
 \\
+  \intotme [\lfbothe
  \nabla_{y}u_{2} \cdot\nu_\eps ]\rteps\di\sigma\di\tau +
  \intotom \lfbothe \nabla_{y}u_{2}\cdot \nabla \rteps\di x\di\tau
 +\const(\delta) + \delta \intotom \rteps^{2}\di x\di\tau
 \\
 +\delta' \intotom \abs{\nabla\rteps}^{2}\di
  x\di\tau +\frac{\const(\delta')}{\eps}
  -\intotom \big\{\dfbothe\frac{\partial u_1}{\partial  \tau}
  (1-\phi_{\eps})\big\}\rteps \di x\di\tau\le
  \\
\const+\const(\delta^{\prime\prime\prime})+
\eps \delta^{\prime\prime\prime}\intotme\rteps^2\di\sigma\di\tau
\\
+\const(\delta^{\prime\prime})+\delta^{\prime\prime}\eps\intotme|\nabla^B\rteps|^2\di\sigma\di\tau
+\delta^{\prime\prime}\eps\intotme\rteps^2\di\sigma\di\tau
 \\
 +\const(\delta^{\prime\prime\prime})+\delta^{\prime\prime\prime} \intotom |\nabla \rteps|^2\di x\di\tau
  +\const(\delta) + \delta \intotom \rteps^{2}\di x\di\tau
 \\
 +\delta' \intotom \abs{\nabla\rteps}^{2}\di x\di\tau
 +\frac{\const(\delta')}{\eps}
 +\const(\delta^{\prime\prime\prime})+\delta^{\prime\prime\prime}
\intotom \rteps^2 \di x\di\tau\,,
\end{multline}
where $\delta^{\prime\prime\prime}$ will be chosen later.
Finally, using Poincar\'{e}'s inequality, Gronwall's lemma and absorbing the gradient term in
\eqref{eq:restenergy_ii} into the left hand side (which is possible choosing
$\delta,\delta^\prime, \delta^{\prime\prime},\delta^{\prime\prime\prime}$ sufficiently small), we get
\begin{equation}
  \label{eq:restenergy_vi}
  \intotom \abs{\nabla \rteps}^{2}\di x\di\tau
  \le \frac{\const}{\eps}\,.
\end{equation}
On recalling the definition of $\rteps$, and invoking again Poincar\'{e}?s inequality, we obtain
\begin{equation}
  \label{eq:L2est_ii}
  \intotom (u_{\eps}-u_{0}-\eps u_{1}(1-\phi_{\eps}))^{2}\di x\di\tau \le
  \const \eps\,.
\end{equation}
Moreover, taking into account that $\reps=\rteps - u_1\phi_\eps$ and using \eqref{eq:restenergy_vi}, it follows that
\begin{equation}\label{eq:a67}
 \intotom |\nabla\reps|^2\di x\di\tau \le \const\left[{\intotom |\nabla\rteps|^2\di x\di\tau
 + \intotom |\nabla (u_1\phi_\eps)|^2\di x\di\tau }\right]\le \frac{\const}{\eps}\,,
\end{equation}
where we recall the estimate for $\nabla (u_1\phi_\eps)$ done in \eqref{eq:lapl_au1phieps}.
Hence, by \eqref{eq:L2est_ii} and \eqref{eq:a67}, we obtain \eqref{eq:error_na}.
Finally, \eqref{eq:error_nb} can be obtained
making use of \eqref{eq:L2est_ii} and taking into account that
\begin{equation*}
  \intotom (\eps u_{1}(1-\phi_{\eps}))^{2}\di x\di\tau \le
  \gamma \eps^{2}\,.
\end{equation*}
This concludes the proof.
\end{proof}



\end{document}